\documentclass[12pt,titlepage, a4paper, draft]{article}

\usepackage{amssymb,latexsym,amsfonts,amsmath,amsthm}
\usepackage{verbatim}

\usepackage{url}

\usepackage{pst-node}  

{\makeatletter \@addtoreset{equation}{section}}

\newcommand{\diag}{\mathop{\mathrm{diag}}\nolimits}
\newcommand{\Ker}{\mathop{\mathrm{Ker}}\nolimits}
\newcommand{\IIm}{\mathop{\mathrm{Im}}\nolimits}

\newcommand{\Inv}[1][\,]{\mathop{\mathrm{Inv}#1}\nolimits}

\newcommand{\N}{{\mathbb{N}}}
\newcommand{\Z}{{\mathbb{Z}}}

\newcommand{\h}{ \mathop{ \mathrm{h} {} }\nolimits }
\newcommand{\e}{ \mathop{ \mathrm{e} {} }\nolimits }

\newcommand{\beq}{\begin{equation}}
\newcommand{\eeq}{\end{equation}}

\newcommand{\bcm}{}

\newcommand{\Ra}{\Rightarrow}

\newcommand{\bpm}{\begin{pmatrix}}
\newcommand{\epm}{\end{pmatrix}}

\newtheorem{theorem}{Theorem}[section]
\newtheorem{lemma}[theorem]{Lemma}

\newtheorem{example}[theorem]{Example}

\newtheorem{prop}[theorem]{Proposition}

\usepackage{color}

\definecolor{ungu}{rgb}{0.8 , 0  , 0.7 }
 \definecolor{exp}{rgb}{0.85,0.7,0}
  \definecolor{msk}{rgb}{0.65,0.3,0.3}

\definecolor{msk2}{rgb}{0.75,0.3,0.1}
\definecolor{r}{rgb}{0.75,0.1,0.4}

\definecolor{exp2}{rgb}{0.85,0.6,0.1}

\definecolor{exp4}{rgb}{0.85,0.3,0.5}

\definecolor{vvs}{rgb}{0.5,0.4,0.7}

\begin{document} 

\title{Linear transformations with characteristic subspaces that are not hyperinvariant}



\author{Pudji Astuti
\\
Faculty of Mathematics\\
 and Natural Sciences\\
Institut Teknologi Bandung\\
Bandung 40132\\
Indonesia
\thanks{The work of the first author was supported by the program  ``Riset 
dan Inovasi KK ITB''  of  the  Institut Teknologi Bandung.}  
         \and
Harald~K. Wimmer\\
Mathematisches Institut\\
Universit\"at W\"urzburg\\
97074 W\"urzburg\\
Germany
}

\date{\today}

\maketitle

\begin{abstract}

\vspace{2cm}

\vspace{.5cm}


\vspace{1cm}

\noindent
{\bf Abstract:}

If  $f$ is  an endomorphism of a finite dimensional vector
space over a field $K$ then 
an invariant subspace $X \subseteq V$ is called
hyperinvariant (respectively, characteristic)  if  
$X$ is invariant under
 all endomorphisms (respectively, auto\-morphisms) 
that commute with $f$.
According to Shoda
(Math.\ Zeit.\ 31,  611--624, 1930)
only if  $|K| = 2$ 
then there exist
 endomorphisms $f$
with invariant subspaces that are  characteristic
but not hyperinvariant. 
In this paper we obtain a description of the
set of all characteristic non-hyperinvariant subspaces 
for nilpotent maps $f$  
with exactly two unrepeated  elementary
divisors.

\vspace{3cm}
\noindent
{\bf Address for Correspondence:}\\
H. Wimmer\\
Mathematisches Institut \\
Universit\"at W\"urzburg\\
Am Hubland\\
97074 W\"urzburg\\
Germany

\vspace{.2cm}
\flushleft{
{\textsf e-mail:}~~\texttt{\small wimmer@mathematik.uni-wuerzburg.de} }\\
{\textsf Fax:}~~
+49 931 8\,88\,46\,11

\flushleft{
{\textsf e-mail:}~~\texttt{\small
pudji@math.itb.ac.id
} }

\end{abstract}

\section{Introduction} \label{sct.int}

Let $K$ be a field, $ V $ an $n$-dimensional vector space
over $K$ and  $f : V \to V $ a $K$-linear map.
A subspace $X \subseteq V$ is said to be {\em{hyperinvariant}}
(under  $f$)  
\cite[p.\ 305]{GLR} 
if it  remains invariant under
all endomorphisms of $V$ that commute with $f$.
If $X$ is an   $f$-invariant subspace of $V$ and if
$X$ is  invariant
under
 all automorphisms of $V$ that commute with $f$,
then \cite{AW4}
  we say that    $X$ is   {\em{characteristic}} (with respect to $f$).
Let $ {\rm{Inv}}( V,f) $, \,$   {\rm{Hinv}}(V,f)$, and
$  {\rm{Chinv}} (V,f)  $ be  sets of  invariant,   hyperinvariant and
  characteristic  subspaces of  $V$, respectively.
These sets are lattices, and
 \[   {\rm{Hinv}} (V,f)
 \subseteq   {\rm{Chinv}} (V,f)   \subseteq   {\rm{Inv}}( V,f)  .
\]
If the   characteristic
  polynomial of $f $ splits over $K$ (such that all eigen\-values of
$f$ are in $ K$)
 then 
  one can restrict the
  study of 
 hyperinvariant and of characteristic  subspaces to the case where
$ f $ has only one eigenvalue, and  to the
case where $f$ is nilpotent.
Thus, throughout this paper we shall assume \,$ f^n = 0 $.
Let 
\,$
\Sigma(\lambda)  =  \diag (1, \dots , 1, 
 \lambda^{t_1}, \dots , \lambda^{t_m}) \in K^{n \times n}[\lambda] $\,
be the Smith normal  form of $f$ such that $ t_1 + \cdots +t_m = n$. 
We say that an  elementary divisor $ \lambda ^r $ is  {\em{unrepeated}}
if  it  appears exactly once in $\Sigma(\lambda) $.

The structure of the  lattice  \,$   {\rm{Hinv}} (V,f) $\,
is well understood (\cite{Lo},  \cite{FHL},
\cite{Lo2},  \cite[p.\,306]{GLR}).
We point out that
 \,$   {\rm{Hinv}} (V,f)  $\, is the 
sublattice of $ {\rm{Inv}}( V,f) $
generated by
\[
  \Ker f ^k,  \,\, \IIm f^k ,  \: k = 0, 1, \dots , n.
\]
It is   known  (\cite{Sh}, \cite[p.\,63/64]{Kap},   \cite{AW4})
that each  characteristic subspace
is hyper\-invariant if  $|K| > 2$.
Hence, only if $V$ is a vector space over the
field $ K = GF(2) $ one may find  $K$-endomorphisms  $f$  of  $V$
with   characteristic subspaces that are not   hyperinvariant.
A necessary and sufficient condition for the existence of
  such mappings $f$ is due to Shoda 
(see also \cite[Theorem~9, p.\,510]{Bae} and  \cite[p.\,63/64]{Kap}).
It involves unrepeated elementary divisors of $f$. 

\begin{theorem}    \label{thm.vnpsa} 
{\rm{\cite[Satz~5, p.\,619]{Sh}}}     
Let $ V $ be a finite dimensional vector space over
  the field  \,$ K =  GF(2) $ and let $ f : V \to V $ be nilpotent.
The following statements are equivalent.
\begin{itemize}
\item[{\rm{(i)}}]
There exists a  characteristic subspace of  \,$V$
 that is not  hyper\-invariant.
\item[{\rm{(ii)}}]
For some  numbers $ R $ and $ S $ with
\,$ R + 1  < S $\,
the map $f$ has 
 exactly one  elementary divisor $\lambda^R $ and exactly one of the
form  $\lambda^S $.
\end{itemize}
\end{theorem}

Provided that 
 $f$  satisfies     condition (ii) of 
Shoda's theorem
 how can one construct 
 all  characteristic subspaces of  $V$ that are not hyper\-invariant?
For the moment the answer to that question is open. 
In this paper we assume 
that $f$ has   exactly one
pair of unrepeated elementary divisors. 
In that case we show how to construct the family of 
characteristic and non-hyperinvariant subspaces associated to  $f$. 
For that purpose   we  prove rather    general  results on the structure of 
characteristic non-hyperinvariant subspaces and  we
clarify the role of unrepeated elementary divisors of $f$. 
We note that our study  can be interpreted in the setting of module theory. 
In the context of abelian group theory
\cite{FuII} one would deal with 
characteristic subgroups of $p$-groups that are not  fully invariant. 

We first discuss an example, which 
displays  features of characteristic
subspaces that will become important later,  and we  introduce 
 concepts that will allow us to  state 
 Theorem~\ref{thm.pfle} at the end of this
section.

\subsection{An example and
basic concepts}  

The  inequality \,$ R + 1  < S $\,
in Theorem~\ref{thm.vnpsa} is valid for 
  $ (R, S)  = (1,3) $. 
In Example \ref{ex.sh} below we describe
 a subspace  that is
characteristic but not hyper\-in\-variant
with respect to 
a map $f$ with elementary divisors $ \lambda$ and $ \lambda ^3$.
We first  introduce some notation, in particular
we define the  concepts of  exponent and height.
We set $  V[ f^j ] =  \Ker f^j $, $j \ge 0 $.
Thus,  $ f^n = 0 $ implies $ V =  V[ f^n]$.
Define $ \iota = {\rm{id}}_V $ and $ f^0 =  \iota$.
Let $ x \in V $. The smallest nonnegative integer $\ell$
 with
$f^{\ell} x = 0$
is  called
the {\em{exponent}} of $x$. We write  $\e(x) = \ell$.
A  nonzero vector  $x $
  is said to have
\emph{height} $q $ if $x \in f^q V$
and $x \notin f^{q+1} V$.
In this case we write $\h(x) = q$. We set $ \h ( 0 ) =  - \infty $.
The $n$-tuple
 \[
H(x) = \big(\h(x), \h(fx), \dots, \h(f^{n-1}x) \big)
\]
is the  {\em{indicator}} \cite[p.\,3]{FuII} or {\em{Ulm sequence}}  \cite{Kap}
of $x$. Thus, if  $\e(x) = k $
 then
 $H(x) = (\h(x),  \dots , \h(f^{k-1}x),  \infty, \dots,  \infty)$.
We say that $H(x) $ has a  {\em{gap}} at $j$, if
 \,$1 \le j  < \e(x)  $\,  and  \,$ h(f^{j}x ) >
1 +  h(f^{j-1}x )$.
Let $  {\rm{End}}(V,f) $  be the algebra of
all endomorphisms of $V$ that commute with $f$.
The group of  automorphisms of $V$
that commute with $ f $ will be denoted by
 ${\rm{Aut}}(V,f) $.
Let
\begin{multline*}
 \langle x  \rangle \,  = \,  {\rm{span}} \{ f^i x , \, i \ge 0 \}
=
\\
\{c_0x + c_1 f x + \cdots + c_{n-1} f^{n-1}x ;
\,
  c_i \in K, \, i = 0, 1, \dots , n-1 \}
\end{multline*}
 be  the  $f$-cyclic subspace  generated by $ x $.  
 If \,$ B  \subseteq V$\, we define
 \,$  \langle B \rangle  =
 \sum _{b \, \in \,  B } \,  \langle \,  b  \, \rangle  $
and 
\[
  B  ^c =  \langle  \alpha b ;  \, b \in B , \,  \alpha \in {\rm{Aut}}(V,f) \rangle .
\]
We call  $ B  ^c $  the 
{\em{characteristic hull}} of $  B  $. 
Clearly, if  $ \alpha \in {\rm{Aut}}(V,f) $ then
  $ \alpha ( f^j  x ) = f^j (  \alpha x ) $  for all
$x \in  V$. Hence  it is obvious
that
\beq \label{eq.aprs}
 \e( \alpha x ) = \e(x) \quad \text{and} \quad \h( \alpha x ) = \h(x)
\quad \mathrm{for \: \: all} \quad  x \in V, \,  \alpha \in {\rm{Aut}}(V,f).
\eeq
Let $ e_1 = (1, 0, \dots , 0) ^T$, $\dots $,  $ e_m =
  (0, \dots , 0, 1)  ^T $  be the unit vectors of $K^m$,
and let $N_m $ 
denote the lower triangular nilpotent $ m \times m $ Jordan block.

\begin{example} \label{ex.sh}
{\rm{\cite[p.\,63/64]{Kap})}
Let $ K = GF(2) $.  
 Consider  $ V = K ^4 $     
 and let $f: V \to V $ 
be $K$-linear with elementary divisors $\lambda $ and $ \lambda ^3$.  
With respect to the basis 
$\{ e_i \}  $, $ i = 1, \dots , 4 $, 
the map $f$ is given by 
$ f x = N x $ with
 \[
N =  \diag(N_1, N_3) =  \left( \begin{array}{c|ccc}
  0 & 0   & 0 & 0
\\ \hline
 0 & 0  & 0  & 0\\
                0 & 1 & 0 & 0
\\
0 & 0 & 1 & 0
\end{array}  \right) ,
\]
and  \,$ V = \langle e_1 \rangle  \oplus  \langle e_2  \rangle
$
with $\e(e_1) = 1 $ and $\e(e_2) = 3$. 
Then   $ \alpha \in  {\rm{Aut}}(V,f) $ if and only if
$ \alpha x = Ax $ and  $A \in K^{4 \times 4 } $
 is a nonsingular matrix satisfying $ AN = NA $,
that is (see e.g.\  \cite[p.\ 28]{ST}), 
\[ 
A =  \left( \begin{array}{c|ccc}
  1 & \nu    & 0 & 0
\\ \hline
 0 & 1  & 0  & 0\\
                0 & \omega   & 1 & 0
\\
 \kappa & \mu  &  \omega   & 1
\end{array}  \right) , \: \kappa, \mu  ,  \nu ,  \omega \in K. 
\]
Define $ z = e_1 + e_3 $.  
We show that  the characteristic subspace 
$X = \langle  z \rangle ^c $ 
is not hyperinvariant. 
We have  
$ \alpha z  =  
 ( e_1 + \kappa e_4) + (e_3 +  \omega  e_4 )   $, 
and therefore 
 \begin{multline*} 
X =
\langle  z \rangle ^c  = 
 \langle  e_1 + fe_2 \rangle ^c = \\
{\rm{span}} \{  e_1 + e_3 ,   e_1 + e_3 + e_4\}  = 
 \{ 0,   e_1 + e_3,   e_1 + e_3 + e_4, e_4 \}   .
 \end{multline*} 
 Let
\, $ \pi _1 = \diag (1, 0 , 0 ,  0) $   be the orthogonal projection
on $ K e_1 $.
Then $  \pi _1  \in  {\rm{End}}(V,f) $.
We have  \,$ \pi_1  z    = e_ 1 $,  but 
$ e_1 \notin  X    $.
 Therefore
  $ X   $ is  a characteristic subspace
that is not  hyper\-invariant.
From 
 $H(z) = (0, 2, \infty, \infty)$
we see that 
 $ H(z)  $ has a gap at $j = 1$.
\hfill $\square$ 
 } 
\end{example}

To   make the  connection with
Kaplansky's exposition of Shoda's theorem
 \cite[p.\,63/64]{Kap}
we  define   the numbers
\[
d(f, r) =
\dim \big(  V[f] \cap
  f ^{r -1} V \,\,  /  \, \, V[f]  \cap  f ^{r}  V \big),
\:\:
 r= 1, 2,  \dots , n.
\]
In accordance with the terminology
of  abelian $p$-groups \cite[p.\,154]{FuI}  or $p$-modules
\cite[p.27]{Kap}
we call $ d(f, r) $ the ($r-1$)-th   {\em{Ulm invariant}} of $f$.
Then 
$   d (f, r) $ is equal to  the number of 
entries $ \lambda ^r $ in the Smith form of  $f$, and 
 $ d(f, r) = 1 $ means that 
  $  \lambda ^r $  is an  unrepeated  elementary divisor. 
In the following it may be convenient to    write $d (r)$ instead   of  $ d (f, r) $.
We recall (see e.g.\  \cite{BCh}) that
methods or concepts  
of abelian group theory must be translated
to modules over principal ideal domains and then specialized
to $K[\lambda]$-modules before they can be applied to linear algebra.

With regard to Theorem~\ref{thm.pfle}  below 
we note additional definitions. 
Suppose  $ \dim \Ker f = m $.  Let $ \lambda ^{t_1} , \dots , \lambda ^{t_m} $
be the elementary divisors of $f$ such that
$ t_1 + \cdots + t_m = \dim V  $. 
Then $V$ can be decomposed into a direct sum of $f$-cyclic  subspaces
$ \langle u_i \rangle $ such that
\beq \label{eq.vdeco}
 V =
  \langle u_1 \rangle
 \,  \oplus \,
\cdots \,  \oplus    \langle u_{ m } \rangle \quad \text{and} 
\quad \e(u_i) = t_i ,  \: i = 1, \dots , m . 
\eeq 
If \eqref{eq.vdeco} holds 
and if the elements in  $U$ are ordered by  nondecreasing exponents
such that
\[
\e( u_1 )  \,\,   \le \,\,  \cdots \,\,  \le \,\,\e( u_m ) 
\]
then we  call  \,$ U = ( u_1, \dots , u_m) $\,    a
{\em{generator tuple}} of $V$
(with respect to $f$). 
The  tuple $(t_m , \dots , t_1) $ of exponents  - written
in nonincreasing order -  is known as {\em{Segre characteristic}} 
of $f$. 
The set of  generator tuples  of $V$ will be denoted by
 $ \mathcal{U} $.
We call $ u \in V $
 a {\em{generator}}  of $V$
(see also \cite[p.4]{FuI})  if 
$u \in U $ for some $U \in  \mathcal{U} $. 
In other words, 
$u \in V $ is a generator if and only if $u\ne0$ and
\beq \label{eq.vpluus}
  V =   \langle u \rangle \oplus V_2  \quad \text{for some}
  \quad V_2 \in  {\rm{Inv}}( V,f) . 
\eeq

If $f$ has only two elementary divisors then 
part (ii)  of  the following theorem gives a description of the 
set ${\rm{Chinv}} (V,f)\setminus  {\rm{Hinv}} (V,f)$.

\begin{theorem}  \label{thm.pfle} 
Assume $ |K| = 2 $. 
Suppose $ \lambda ^R $ and $ \lambda  ^S $ are unrepeated elementary
divisors of $f$ and 
$ R + 1 < S $.  Let  $ u $ and $ v $ be
corresponding  generators of $V$ such that
$\e(u) = R $   and $ \e(v ) = S $. 
\begin{itemize}
\item[{\rm{(i)}}]
A   subspace  
 \beq \label{eq.zbns}
  X =   \langle f^{R -s } u  +  f^{S -q } v \rangle ^c
\eeq
is charac\-teristic and not hyperinvariant
if  the integers $s, q $ satisfy  
\beq \label{eq.vrschv}
  0 < s \le R ,  \: \: s < q ,  \:\:  R -s  < S -q  . 
\eeq 
\item[{\rm{(ii)}}]
Suppose 
$ V = \langle u \rangle \oplus \langle v \rangle $. 
Then  an invariant subspace $X \subseteq V $
is charac\-teristic and not hyperinvariant 
 if and only if $X$ is of the form \eqref{eq.zbns} 
and $s,q $ satisfy  \eqref{eq.vrschv}. 
\end{itemize} 
\end{theorem}

The proof of  Theorem~\ref{thm.pfle}(ii)  will
be given in Section~\ref{sct.twoc}, where 
two propositions will be proved, one dealing with  sufficiency 
 and the other one with necessity of condition  \eqref{eq.vrschv}. 
In Section~\ref{sct.eg}
we split the space $V$ into two complementary invariant 
subspaces $E$ and $G$ such that
the unrepeated elementary divisors of $f $ are those 
of $ f_{|E} $ and the repeated ones are those of $  f_{|G}  $.
It will be shown that 
a characteristic subspace $X $ is  hyperinvariant   in $V$ 
if and only if   $X \cap E $  is  hyperinvariant in $E$.
An application of that approach is a proof of 
Theorem~\ref{thm.pfle}(i).  
In Section~\ref{sct.mn} 
we extend Theorem~\ref{thm.pfle}(ii)  assuming 
that $f$ has only two unrepeated elementary divisors.

\section{The case of two elementary divisors} 
\label{sct.twoc} 

In this section we give a proof of Theorem~\ref{thm.pfle}(ii).
It is based on auxiliary results on hyperinvariant subspaces 
and on images of generators under $f$-commuting 
automorphisms of $V$.

\subsection{Hyperinvariant subspaces} 

Suppose $\dim \Ker f = m $.
Let  $ U = (u_1, \dots , u_m) \in \mathcal{U} $ be a 
generator tuple with $ \e(u_i ) = t_i$, $ i = 1, \dots , m$,
such that 
\[
 0 < t_1 \le \cdots \le t_m .
\] 
Set  $ \vec{t}  = (t_1 , \dots , t_m) $ and $t_0 = 0$. 
Let  $\mathcal{L}( \vec t \, ) $  be the set of $m$-tuples 
$\vec r = (r_1 , \dots , r_m )  \in \Z^m $\,  
satisfying
\beq \label{eq.zwrug} 
0 \le r_1 \le \cdots  \le r_m  \:\:\,  {\rm{and}} \:\:\, 
0 \le t_1 - r_1   \le \cdots \le t_m -  r_m  .
\eeq 
We write $ \vec r   \preceq \vec s $
if  $ \vec r= ( r_i )_{i=1} ^m $, $ \vec  s = (s_i ) _{i=1} ^m $ 
$ \in \mathcal{L}(\vec t \, \, ) $
 and
 $r_i \le s_i$, $i= 1, \dots , m $.
Then  $\big( \mathcal{L}(\vec t \, ),   \preceq  \! \big) $ is a
 lattice.  The following theorem is due to Fillmore, Herrero and 
Longstaff \cite{FHL}.  We  refer to \cite{GLR} for a proof.

\begin{theorem}  
 \label{thm.knwr} 
Let $f : V \to V $ be nilpotent. 
\begin{itemize} 
\item[\rm{(i)}]
If $ \vec r \in  \mathcal{L}(\vec t \,) $,   
then
\[ 
W(\vec r \, ) = f ^{r_1}V \cap V[ f^{t_1 - r_1} ]
 \,  + \cdots + \, 
 f^{r_m}V \cap V[ f^{t_m- r_m} ] 
\]
 is a hyperinvariant subspace.
Conversely, each
$ W \in  {\rm{Hinv}}(V,f)$ is of the form
$W = W(\vec r \, ) $ for some $ \vec r \in  \mathcal{L}(\vec t \, ) $.
\item[\rm{(ii)}] 
 If  $ \vec r \in  \mathcal{L}(\vec t \, ) $
then
\,$ W(\vec r \, ) =
  f ^{r_1}  \langle u_1 \rangle  \,  \,  \oplus  \cdots  \,  \,
\oplus  
f ^{r_m}  \langle  u_m \rangle  $. 
\item[\rm{(iii)}]   The mapping  $ \vec r \mapsto W( \vec r \, ) $
is a lattice isomorphism from $\big( \mathcal{L}(\vec t \, \, ),  \preceq  \! \big) $
  onto $( {\rm{Hinv}} (V,f), \supseteq  )$.
\end{itemize} 
For a given $\vec t $  
the number of  hyperinvariant 
subspaces is 
\beq   \label{eq.nhnu} 
  n_H(\vec t \, ) = \prod \nolimits _{i = 1} ^m 
(  1 + t_i - t_{i-1}) . 
\eeq
\end{theorem}

Let $X_H$ denote the 
largest hyperinvariant subspace contained in  a characteristic subspace
$X$.  Using a generator tuple  
  $ U = (u_1, \dots , u_m)   \in \mathcal{U}$
one can give an explicit description of 
$X_H $. 
Let   $x \in V $  be decomposed as 
 \beq \label{eq.smxju}
x = x_1 + \cdots + x_m ,   \: x_i \in \langle u_i \rangle, \; i = 1,
\dots ,m, 
\eeq 
and let  \,$ \pi _{j} : V \to V $ be  the projections defined by $\pi_j x = x_j   $,
$ i = 1, \dots ,m$. 
 If  $X  \subseteq V $  
then  $\pi_j x   \in X $ for all $x \in X $ is equivalent to
\beq \label{eq.vlfgin}
\pi _j X = X \cap \langle u_j \rangle . 
\eeq 
The following theorem shows that  \eqref{eq.vlfgin} holds  
if $\e(u_j) = t $ and  $ \lambda ^t $ is a repeated elementary divisor.

\begin{theorem}  \label{thm.wogn} 
{\rm{\cite[Lemma 4.5, Lemma 4.2, Theorem 4.3]{AW5}}}
Suppose  $X $ is a characteristic subspace of $V$.
Let  $ U = (u_1, \dots , u_m) \in \mathcal{U}$.
\begin{itemize} \item[\rm{(i)}]
 If $d(t_j ) > 1 $ then 
$\pi_j X = X \cap \langle u_j \rangle $.
\item[\rm{(ii)}]
The subspace $X$ is hyperinvariant if and only if
\beq \label{eq.adin}
\pi_j X = X \cap \langle u_j \rangle, \: j = 1,  \dots, m,
\eeq
or equivalently,
\beq \label{eq.brhy}
X =  \oplus _{i=1} ^m \,  \big( X \cap \langle u_i  \rangle \big). 
\eeq
\item[\rm{(iii)}]
The subspace 
\beq \label{eq.larg}
X_H  = \oplus _{i=1} ^m \, \big( X \cap \langle u_i  \rangle \big)
\eeq
is the largest hyperinvariant subspace contained in $X$.
\end{itemize}
\end{theorem}


In a characteristic subspace
 $X $  elements outside of $ X_H$ 
are of special interest (if they exist).

\begin{lemma}  \label {la.gfnc}  
Let   $X $  be a characteristic subspace. 
\begin{itemize} 
\item[\rm{(i)}]
If  $ W $ is a proper  subspace of $X$
then  $X = \langle  X \setminus W \rangle ^c $.
\item[\rm{(ii)}] 
If  $ X $ is not hyperinvariant then
 $X = \langle  X \setminus X_H \rangle ^c $
and
\beq \label{eq.xglxc}
 \langle  X \setminus X_H \rangle ^c = \langle  X \setminus X_H
\rangle .
\eeq
\end{itemize} 
\end{lemma}

\begin{proof}
(i) 
From
 $ X \setminus W   \subseteq X $
and $ X ^c = X $ follows
 $  \langle  X  \setminus W  \rangle ^c   \subseteq  X $.
Conversely,
if $ x \in  X $ then either  $ x \in  X \setminus W $
or $x \in W $.
In the first case it is obvious that
$ x \in  \langle   X \setminus W \rangle ^c $.
Suppose $ x \in  W  $.  Choose an element  $ z \in  X \setminus W $.
Then $ x + z  \in  X \setminus W  $.
Thus  $  z  \in  \langle   X \setminus W \rangle ^c  $ and 
$ x + z  \in  \langle   X \setminus W\rangle ^c  $, 
and therefore 
$ x \in  \langle   X \setminus W \rangle ^c  $.
Hence
$ X \subseteq   \langle   X \setminus W \rangle ^c  $,
which completes the proof.

(ii) 
Because of  $ X \supsetneqq X_H  $
 we can choose   $W =  X_H  $,   and obtain  
$X = \langle  X \setminus X_H \rangle ^c $. 
Let us show that
\beq  \label{eq.amb}
\alpha  (  X \setminus X_H ) =  X \setminus X_H
\: \: \,  
\text{for all} \: \: \,  \alpha   \in {\rm{Aut}}(V,f).
\eeq
Since $X_H$ is hyperinvariant and $X$ is characteristic
we  have   $  \alpha  ( X_H ) =  X_H $
and   $ \alpha  (X) = X $.
Consider  \,$ x \in X \setminus  X_H $.
Suppose   $ \alpha   x \in  X_H $.
Then  $ x \in \alpha ^{-1} ( X_H  ) =  X_H  $,
which is a contradiction.
It is obvious that \eqref{eq.amb} is equivalent to \eqref{eq.xglxc}.
\end{proof}

\subsection{Images under automorphisms} \label{sbst.imgn}

If $x \in V $ and   $ \alpha \in  {\rm{Aut}}(V,f) $ 
then 
it follows from  \eqref{eq.aprs} 
that  $ H( \alpha x ) =  H( x )  $. 
We note a converse result due to Baer. 

\begin{theorem}   \label{thm.cru}
{\rm{(See \cite{Kap}, \cite[p.\,4]{FuII})}}
   Let $x, y \in V$.
Then
$ H(x) = H(y) $ if and only if
$ y = \alpha x $ for some
 $ \alpha \in  {\rm{Aut}}(V,f) $.
\end{theorem}
We shall use  Baer's theorem  in Lemma~\ref{la.unrpes}
 to determine the set 
$ \{\alpha u ; \alpha \in  {\rm{Aut}}(V,f)  \} $
for   generators  $ u $ of $V$.
With regard to the proof of Lemma~\ref{la.unrpes} 
we put together basic facts on  exponent and height.


If  $ x \in V$,  $x \ne 0$,  and $ \e(x) = k $,   then
$  \e(f^j x ) = k - j $,  $j = 0, 1, \dots, k  -1$.
The height of $  f^j x  $ satisfies the inequality
$
\h(f^j x ) \ge  j +  \h(x) $,  
 $ j = 0, 1, \dots, k  -1$. 
If $x_1, \dots , x_m \in V $ then  
\beq \label{eq.hmn} 
\h(x_1 + \cdots + x_m) \ge
\min\{ \h(x_i) ; 1 \le i \le m \} .
\eeq
In general,
the inequality \eqref{eq.hmn} is strict. 
Consider Example  \ref{ex.sh}
with
$ x_1 = e_3 + e_1 $, $x_2 = e_4 + e_1 $, and
 $ \h(  x_1 ) =  \h(  x_2 ) = 0 $ and $ \h(x_1 + x_2 ) = 1$. 
We have equality in \eqref{eq.hmn} if the vectors $x_i$ satisfy the 
assumption of the following lemma.

\begin{lemma} \label{la.mxmn}
Let
\,$ U = ( u_1, \dots , u_m) \in  \mathcal{U} $.
If
\,$
  x = \sum  \nolimits _{i = 1} ^m\, x_i $, \,$x_i \in \langle u_i \rangle $,
$i = 1, \dots , m$, $x \ne 0 $,
then 
\beq \label{eq.hegl}
 \h(x) = \min \{ \h(x_i);  \: 1 \le i \le m ,\, x_i \ne 0 \}
\eeq
and  
\beq \label{eq.epngl}
 \e(x) = \max \{ \e(x_i);  \: 1 \le i \le m ,  \, x_i \ne 0 \}.
\eeq
\end{lemma}

\begin{proof}
To prove the identity \eqref{eq.hegl}  we set
$  \tilde{q}  =   \min  \{ \h(x_i); \, x_i \ne 0 \}  $ and  $q =\h(x )$.
Then $ q \ge \tilde{q} $.  On the other hand we have
$ x = f^q y$, $ y = \sum y_i $, $ y_i \in  \langle u_i \rangle $.
Hence $ x_i = f^q y_i $ for all $i$, and therefore $  \tilde{q} \ge q $.

With regard to \eqref{eq.epngl}  we
set  $ \tilde{\ell}  =    \max   \{ \e(x_i); \,  x_i \ne 0 \}$  
and $\ell =  \e(x) $.
Then $0 = f^{\ell} x = \sum _{i = 1} ^m f^{\ell}  x_i $
implies $  f^{\ell}  x_i = 0 $ for all $i$. Hence
$  \tilde{\ell}  \le \ell$.
On the other hand
\[
0 \ne f^{ \ell- 1}  x = \sum \nolimits  _{i = 1} ^m f^{\ell -1}  x_i
\]
implies
 $ f^{  \ell- 1}  x_j \ne 0 $ for some $j$, $1 \le j \le k$,
and therefore
$  \tilde{\ell} \ge \ell$.
\end{proof}

We remark  that
the preceding lemma  can be deduced from
 results on  marked subspaces in \cite{Bru}.


\begin{lemma} \label{la.jtgn} 
Suppose $\lambda  ^t $ is an elementary divisor of $f$. 
Then  
$ u $ is a generator of $V$ with $ \e(u) = t $ if and only if  $ f^t u = 0$
and 
\beq \label{eq.zmhm} 
 \h(f^j u ) = j ,  \:  j = 0, 1 , \dots , t-1 .
\eeq
\end{lemma}

\begin{proof} 
Suppose   $ u $ is a generator  and 
\beq \label{eq.sinw} 
V = \langle u  \rangle \oplus V_2 \quad \text{and} \quad  \e(u) = t .
\eeq 
Let 
$ \h(  f ^{t-1} u ) = (t-1) + \tau $, $ \tau \ge 0 $. Then  
$  f ^{t-1} u =  f ^{t-1 + \tau} \tilde{w} $
for some $ \tilde{w} = w + w_2$ with $ w \in  \langle u  \rangle $, 
$ w_2 \in  V_2 $.  Then $ \langle u  \rangle \cap V_2 = 0 $ 
implies  $ f ^{t-1} u =  f ^{t-1 + \tau}  w $, and we obtain $\tau  = 0 $.
Hence $ \h(f^{t-1} u ) =  t-1 $, which  is 
 equivalent to \eqref{eq.zmhm}.

 Now  suppose $ u \in V $ satisfies  
\eqref{eq.zmhm}.
Let $ v \in V $ be a generator corresponding to $ \lambda ^t $
such that $ V = \langle v \rangle \oplus W_2 $ and $ \e(v) = t $. 
We have shown before that 
$ H(v) = (0, 1,  \dots, t-1, \infty , \dots , \infty) $.
Hence 
 $ H(u ) = H(v) $, and therefore 
Theorem~\ref{thm.cru} implies 
$ u = \alpha v $ for some $ \alpha \in  {\rm{Aut}}(V,f) $. 
Then  
$ V = \alpha (    \langle  v  \rangle \oplus  W_2 )  = 
 \langle u  \rangle \oplus  \alpha  W_2  $ shows that 
$ u $ is also a generator. 
\end{proof}

A  consequence of Lemma~\ref{la.jtgn}  is the following observation. 
Suppose  $u$ is a generator of $V$ and $ w \in  \langle u \rangle $
and $ \h(w) = 0 $.
Then  $  \langle w  \rangle  =  \langle  u  \rangle $.

\begin{lemma} \label{la.unrpes} 
Suppose $ \lambda ^t $ is an unrepeated  elementary divisor of $f$.
Let  $U =  (u_1, \dots , u_m) \in \mathcal{U} $, and $ \e(u_p) = t $.
If  $ u \in V $ then the  following statements are equivalent.
\begin{itemize}
 \item[\rm{(i)}]
The vector $ u $ is a generator of $V$ with $\e(u) = t $.
 \item[\rm{(ii)}] 
There exists an  $  \alpha \in  {\rm{Aut}}(V,f)  $ such that 
$ u  = \alpha u_p $. 
 \item[\rm{(iii)}] 
We have 
\begin{multline}  \label{eq.nrudx}
u = y +  g  \:\: 
 with \:\:  y = c u_p + f v,  \: c \ne 0 , \: v \in \langle u_p  \rangle, 
\\  \: g \in
 \langle u_i ;  \;   i \ne p  \rangle [f ^{t}  ]  . 
\end{multline}
\end{itemize}
\end{lemma}

\begin{proof} 
If $u$ and $ u_\rho $ are generators of $V$ 
then  we have $ \e(u) = e(u_\rho)$  if and only if 
$ H(u) = H( u_\rho) $.  Hence 
  it follows from Theorem~\ref{thm.cru} and Lemma~\ref{la.jtgn}
 that (i) and (ii) are equivalent. 

(i) $\Ra $ (iii) Let $u$  be decomposed such that   $ u  = y + g $ 
and   $y \in \langle u_p  \rangle $ and 
$ g \in
 \langle u_i ;  \;   i \ne p  \rangle    $.  
Then \eqref{eq.epngl} implies 
$ \e(g)  \le  \e(u) = t $, that is  $ g \in V[f^{t}]$.
Let us show that
$\h(y) = 0 $, or equivalently 
\beq \label{eq.xtoa} 
  y = c u_p + f v,  \: c \ne 0 , \: v \in \langle u_p  \rangle . 
\eeq 
We have $  g = g_{<} + g_{>} $ with
$   g_{<} \in   \langle u_i ; \,  i < p \rangle $,
and  
$g_{>}  \in  
   \langle u_i ; \,   i > p \rangle  $.  
If $ i < p $  then $ \e(u_i) < \e(u_p) =  t $. 
Hence  $ f^{t -1 }  g_{<} = 0 $, and 
\beq \label{eq.neigf}
 f^{t -1 } u  =    f^{t-1 }  y +   f^{t -1 }    g_{>}   .
\eeq 
If  $ i > p $   then   \,$  \e(u_i) > t $. 
Therefore 
\[
  \langle u_i \rangle[  f^{t}   ] = f^{ \e(u_i)  - t}  \langle u_i\rangle
\subseteq f V .
\]
Hence    $\h(  g_{>}    )  \ge 1$ and     \,$\h(f^{t -1}     g_{>}     ) \ge  t $. 
Now suppose  $ \h(y) \ne 0 $. 
Then $ \h(y) \ge 1 $  and  \mbox{$  \h (  f^{t-1 }  y ) \ge t $}. Therefore
\eqref{eq.neigf} implies 
$\h( f^{t -1 } u ) \ge t  $.
This is a contradiction to the  assumption that $ u $ is a generator with 
 $ \h( f^{t -1 } u )  = t-1 $.  Hence $ \h(y) = 0 $. 

(iii) $\Ra $  (i) 
Assume 
\eqref{eq.nrudx}.
Then 
\eqref{eq.xtoa}   
implies that $y$ is a generator with 
$ \e(y) = t $. Thus
 $ \h(y) = 0 $ and $ \h( f^{t-1} y)  = t - 1 $.  Moreover, we have  $ \e(g) \le t $. 
Hence  $ \e(u) =  \max \{ \e(y) ,  \e(g) \}   =  t $.
From \eqref{eq.hegl} follows  
\[
 t-1 \le \h( f^{t -1} u ) =  \min \{ \h( f^{t-1} y) , 
 \h( f^{t-1} g) \}  \le  \h( f^{t-1} y)  = t - 1 .
\]
Hence $  \h( f^{t-1} u)  = t - 1  $.  Then Lemma~\ref{la.jtgn} 
completes the proof. 
 \end{proof}

Lemma  \ref{la.unrpes} 
 will be used in the proof of  Proposition \ref{la.rgwpo}. 
We note that the assumption   $|K| = 2 $  implies that the vector 
$y$ in \eqref{eq.nrudx} 
is of the form 
\beq  \label{eq.spk2} 
y =  u_{\rho} + fv , \; v  \in \langle u_\rho \rangle . 
\eeq 

\medskip 

If $u$ is a generator with $ \e(u) = t$ then there is no gap in 
the indicator sequence 
$H(f^j u ) = ( j , j +1 , \dots , t-1, \infty ,\dots , \infty) $. 
We mention without proof that
$ \langle x \rangle ^c $ is hyperinvariant if and only if
$H(x) $ has no gap. 
We only need the following special case of that 
result.

\begin{lemma} \label{la.dinch} 
Let $\lambda ^t $ be an unrepeated elementary divisor of $f$
and   $ u $   a generator of $V$ with $\e(u) = t $.
Then 
\beq \label{eq.glaml} 
\langle f^j  u \rangle ^c = \IIm f^j \cap  \Ker f^{t -j} , \:  j = 0 ,
\dots , t ,
\eeq
and  $ \langle f^j  u \rangle ^c $ is hyperinvariant. 
\end{lemma} 

\begin{proof} 
Let $U =  (u_1, \dots , u_m) \in \mathcal{U} $ such that   $ \e( u_p ) = t $
and $ u_p  = u$. 
Then 
 \[
 \Ker f ^t  =   \langle  u_p  \rangle
 \bigoplus  \big( \oplus _{1 \le i \le m ; i \ne p }   
 \langle  u_i  \rangle  [f^t]  \big) 
\]
and 
$ \langle  u _p  \rangle  =  \big\langle c  u _p  + f  v  ; c \ne
0 ,\; v \in \langle  u _p  \rangle    \big\rangle  $.
Hence 
\begin{multline*} 
\Ker f ^t  =  \langle  c  u _p   + f  v + g ;   c \ne
0 ,\; v \in \langle  u _p   \rangle , \: g \in  \langle  u _i ; i
\ne p \rangle [f ^t ] \rangle = \\ 
 \langle  \alpha u_p ;  \alpha \in    {\rm{Aut}}(V,f)  \rangle
=   
\langle  u_p  \rangle ^c . 
\end{multline*}
 If  $0 < j < t$ 
then 
\[
\langle  f^j   u _p  \rangle ^c =  f^j \langle     u _p  \rangle ^c
= 
f^j \Ker f^ t = f^j \Ker f^{ t -j } = 
\IIm f^j \cap  \Ker f^{ t -j } .
\]
\end{proof}

A general theorem  that contains the following lemma can 
be found in  \cite[Lemma 65.4, p.\ 4]{FuII}. 

\begin{lemma} \label{la.baer}
Suppose 
$V =  \langle u_1 \rangle \oplus  \langle u_2 \rangle $
and $ \e(u_1) < \e(u_2)$.
If  $x \in V $ then there exists an automorphism $ \alpha \in    {\rm{Aut}}(V,f)$
such that 
\beq \label{eq.brap} 
 \alpha x =  f^{k_1} u_1 +  f^{k_2} u_2 
\eeq 
for some $k_1, k_2 \in \N_0$.
\end{lemma}

\begin{proof}
Let $x = x_1 + x_2 $, $ x_i \in   \langle u_i \rangle $, $ i = 1,2$.
Suppose 
$ x_1 \ne 0 $ and $ x_2 \ne 0 $.
If 
$\h(x_i) =k_i   $ 
then   $x_i = f^{k_i} \tilde{u}_i $, $ \tilde{u}_i \in    \langle u_i \rangle $,
$ \h( \tilde{u}_i) = 0 $,  $ i = 1,2$. 
Therefore $  \langle \tilde u_i \rangle  = \langle  u_i \rangle $, $ i = 1,2$,
such that 
\mbox{$ (  \tilde{u}_1,   \tilde{u}_2 ) \in \mathcal{U} $}.
Then 
$ \alpha :   (  \tilde{u}_1,   \tilde{u}_2 )    \mapsto 
 (  u_1,   u_2 )  
 \in    {\rm{Aut}}(V,f)$
yields 
\eqref{eq.brap}.  
\end{proof}

\subsection{Proof of  Theorem~\ref{thm.pfle}(ii)}

In this section we assume  $\dim \Ker f = 2$ such that
 \,$ V = \langle u _1 \rangle \oplus  \langle  u_2 \rangle $.
 We prove two propositions. They 
  provide the complete description of the
set   ${\rm{Chinv}} (V,f) \setminus  {\rm{Hinv}} (V,f)$
in Theorem~\ref{thm.pfle}(ii). 
The following notation will be convenient. 
If we  write
\[
\alpha :  (u_1, u_2)  \mapsto (\hat u _1 , \hat u _2  ) \in {\rm{Aut}}(V,f) 
\]
we assume  $ (\hat u _1 , \hat u _2  ) \in \mathcal{U} $, 
and  
$ \alpha $ 
denotes the  automorphism 
in $ {\rm{Aut}}(V,f) $
 defined by $ ( \alpha u_1,  \alpha  u_2)  =  (\hat u _1 , \hat u _2  )$.

\begin{prop}  \label{la.allen}
Let $|K| = 2 $. 
Suppose $V = \langle u_1\rangle \oplus \langle   u_2 \rangle $ and $\e(u_1) = R$,
$\e(u_2) = S$, 
 and $ R + 1 < S $.
If  $ X $
is a characteristic non-hyperinvariant subspace of  $V$
then 
\begin{multline}  \label{eq.chldr}
X = \langle  f^ {R - s} u_1  + f^ {S -q  } u_2 \rangle ^c  \quad 
with
 \\
 \quad 0<  s< q \quad \hbox{and} \quad 0 \le  R - s < S - q , 
\end{multline} 
and 
\beq \label{eq.hxnru} 
X_H =  \langle  f^ {R - s+ 1} u_1  ,  f^ {S -q +1 } u_2 \rangle
= 
\IIm  f^ {R - s+ 1} \cap \Ker f^{ q - 1 } 
\eeq 
is the largest hyperinvariant subspace contained in $X$. 
\end{prop}

\begin{proof} 
Theorem~\ref{thm.wogn}(iii)  and Theorem~\ref{thm.knwr}(ii)  imply 
  \beq \label{eq.nxhm} 
  X_H  = (   X \cap \langle u_1 \rangle ) 
\oplus (  X \cap \langle u_2\rangle )
= W(\vec r \, )    
= \langle  f ^{r_1} u_1 \rangle   \oplus  \langle  f ^{r_2}   u_2 \rangle 
\eeq
for some pair  $\vec r =  (r_1, r_2) $ satisfying
$r_1 \le R $, $ r_2 \le S$,  and
\beq \label{eq.dnhy}  
0 \le r_1 \le r_2  \quad {\rm{and}} \quad
0 \le R - r_1  \le S - r_2 .
\eeq
Since $    X $ is not hyperinvariant we have $    X \supsetneqq   X_H  $.
Let $x \in     X \setminus    X_H $. 
By  Lemma~\ref{la.baer}
there exists an
automorphism $\alpha \in {\rm{Aut}}(V,f) $  be such that
\beq \label{eq.rlops} 
\alpha x  =   z =   
 f^{\mu_1} u_1 + f^{\mu_2}  u_2 \;  \in  \;    X \! \setminus \!    X_H . 
\eeq 
Suppose
 $ \mu_1 \ge r_1 $.  Then
\,$ f^{\mu_1} u_1 \in  \langle f^{r_1} u_1    \rangle  
 \subseteq   X $, 
and because   $  z \in    X$  also    $  f^{\mu_2} u_2  =  z - f^{\mu_1} u_1   \in    X $. 
Hence $ f^{\mu_1} u_1 \in    X \cap \langle u_1 \rangle $ and 
 $ f^{\mu_2} u_2 \in    X \cap \langle u_2 \rangle $, 
and we would obtain $ z \in X_H$.   Similarly, it is impossible that
$ \mu_2 \ge r_2$.
Therefore
\beq \label{eq.rqbd} \mu_1 < r_1 ,  \: \: \mu_2 < r_2 .   \eeq
We shall see that  
\beq \label{eq.vrlt}
  \mu_1 + 1 =  r_1 , \quad   \mu_2  + 1=  r_2  .
\eeq
If  $R = 1$ then $\mu_1 = 0 $,  
and $\mu_1 + 1 =  r_1 =1 $.
If $R > 1 $ then  $f u_1 \ne  0  $, and  
\[
\beta : (u_1 , u_2) \mapsto (u_1 + fu_1, u_2)   \in {\rm{Aut}}(V,f) 
\]
yields
 \[
 \beta   z = 
( f^{\mu_1}u_1 +  f^{\mu_1 +1}u_1)  +  f^{\mu_2}u_2    =    z +   f^{\mu_1 +1}u
_1 \in   X.
\]
Since $   X $ is characteristic and 
  $  z  \in    X $, we have  $ \beta  z \in    X$. Hence    
\[
 f^{\mu_1 +1}u_1  \in    X \cap \langle u_1 \rangle =  
\langle  f^{r_1}  u_1 \rangle ,
\]
and we obtain  $ \mu_1  + 1 \ge r_1 $.
 A similar argument yields,  $ \mu_2  + 1 \ge r_2 $. 
Hence \eqref{eq.rqbd} implies the relations  \eqref{eq.vrlt}.
Thus
\[
  z =  f^{\mu_1 }u_1 + f^{\mu_2 }u_2,  
\: \mu_1 = r_1 - 1,  \; \mu_2 = r_2 -1 .
\]
Hence  there exists a unique
 vector $z \in    X \setminus     X_H$
with a representation  \eqref{eq.rlops}. 
Then \eqref{eq.amb} 
implies 
$  X \setminus  X_H = \{ \alpha z ;  \alpha \in  {\rm{Aut}}(V,f) \} $ 
and \eqref{eq.xglxc}  yields 
\beq \label{eq.ksptu} 
X =  \langle   z \rangle ^c =  \langle    f^{\mu_1}  u_1 +   f^{\mu_2 }  u_2\rangle^ c .
\eeq 
 According to the definitions of $s$ and $t$ we have 
\,$ 0 \le \mu_1 \le \mu_2$
 and  \,$0 \le  R- \mu_1 \le S - \mu _2 $. 
Hence it remains to show that  
\beq  \label{eq.cnglu} 
\mu_1 \ne R , \:  \mu_1 \ne \mu_2 \:\: {\rm{and}}  \:\:  
R - \mu_1 \ne S - \mu_2.
\eeq 
Suppose $ R = \mu _1$.  Then $z = f^{\mu_2}u_2 $,
and  Lemma~\ref{la.dinch} implies $ X = 
\langle    f^{\mu_1}  u_2 \rangle ^c  \in {\rm{Hinv}}(V, f)  $.
Suppose $ \mu _1 = \mu_2  $.
Then  $  z =  f^{\mu_1} (u_1 + u_2 )  $. 
Using  
\[
\gamma : (u_1, u_2 ) \mapsto ( u_1,   u_2 +  u_1 ) \in  {\rm{Aut}}(V,f) 
\]
we obtain  
$ \gamma ^{-1}  z =  f^{\mu_1}  u_2 \in  X $.
Hence Lemma~\ref{la.dinch} implies 
$
 X = \langle    f^{\mu_1}  u_2 \rangle ^c \in {\rm{Hinv}}(V, f)  $.
Suppose $R- \mu_1 =  S - \mu _2 $. Then 
$ S - (\mu_2 - \mu_1) = R $ and 
\[
z = f^{\mu_1} ( u_1 + f^{\mu_2 - \mu_1} u_2 )  =  f^{\mu_1} ( u_1  + f^{S - R } u_2 ) .
\]
Therefore 
\,$
\sigma :  (u_1, u_2 ) \mapsto ( u_1  + f^{S - R } u_2 ,   u_2 ) \in  {\rm{Aut}}(V,f)  
$\, 
yields 
$ \sigma ^{-1} z = f^{\mu_1} u_1 $.
Then
$X = \langle    f^{\mu_1}  u_1 \rangle ^c  \in {\rm{Hinv}}(V, f)  $.
Hence the inequalities \eqref{eq.cnglu} are valid.

From \eqref{eq.nxhm}  
and  \eqref{eq.vrlt}
follows 
$ X_H =  \langle    f^{\mu_1+ 1}  u_1 ,  f^{\mu_2 + 1}  u_2\rangle $.
Moreover 
\eqref{eq.cnglu} implies 
$  \langle    f^{\mu_1+ 1}  u_1 ,  f^{\mu_2 + 1}  u_2\rangle 
=
\IIm f ^{\mu_1 + 1 } \cap \Ker f^{S -(\mu_2 + 1)} $. 

We obtain \eqref {eq.chldr}
and \eqref{eq.hxnru} 
if we 
set $  \mu_1 =  R - s $ and $ \mu_2 = S -q $. 
\end{proof}

\begin{prop} \label{la.rgwpo}  
Assume $|K| = 2 $. 
Suppose $V = \langle u_1\rangle \oplus \langle   u_2 \rangle $, and 
$\e(u_1) = R$,
$\e(u_2) = S$ such that 
$ R + 1 < S $.
Let $s, q$ be integers satisfying 
\beq \label{eq.vrzs}
0 <  s < q , \: \:   0 \le R -s  < S -q .  
\eeq 
Then  the subspace 
\[
X =\langle f^{R -s } u_1 + f^{S -q } u_2  \rangle ^c 
\]
is characteristic and not hyperinvariant, 
and 
 \beq  \label{eq.sncw} 
X   = \langle  f^{R -s } u_1 + f^{S -q } u_2 ,  \;
 f^{R -s  +1} u_1 , \; f^{S -q  + 1} u_2 \rangle .  
\eeq
We have 
\beq \label{eq.dmax}
 \dim X = s + q -1 . 
\eeq
If $s > 1 $ then 
$f_{|X} $ has the  elementary divisors $ \lambda ^q $ and 
$\lambda ^{s - 1} $. If $s = 1 $ then 
$X =  \langle   f^{R -s } u_1 + f^{S -q } u_2  \rangle $
and the corresponding elementary divisor is   $ \lambda ^q $.   
\end{prop}
\medskip

\begin{proof} 
Define 
$z =  f^{R -s } u_1 + f^{S -q } u_2 $.
Then $ X = \langle  z  \rangle ^c $.
Set 
$ \tilde{q} = R + ( q - s) $ 
and  
\[
 \tilde{z} =  u_1 + f^{S -\tilde q} u_2  \quad \text{and} \quad 
\tilde X =  \langle  \tilde{z}  \rangle ^c .
\] 
Then 
\,$ 
(R- s) + (S - \tilde q ) = S - q$.
Therefore $ z = f ^{R-s }   \tilde z $ and $ X =  f^{R-s }  \tilde X $, 
and \eqref{eq.vrzs} is equivalent to
\beq \label{eq.turd} 
 R <  \tilde q    ,   \:   \:  0 <  S - \tilde q  .
\eeq 

Let us first deal with  the height-zero space
$ \tilde X =   \langle  \tilde{z}  \rangle ^c $
and then pass to  $  X =   \langle  z   \rangle ^c $. 
 Let $ \alpha    \in {\rm{Aut}}(V,f) $.
We determine  $ \alpha  \tilde{z} $ using  Lemma~\ref{la.unrpes}. 
Recall  that $ | K | = 2 $ implies \eqref{eq.spk2},  that is 
we have 
$ y = u_\rho + fv $, $ v \in \langle   u_\rho   \rangle $ in 
 \eqref{eq.nrudx}.
If
$\alpha  u_1   =  x_1 + x_2  $,
\,$x_i \in
\langle u_i \rangle$, $ i = 1, 2 $, 
 then 
\[
 x_1=    u_1+ f v_1, \,
 v_1  \in  \langle u_1 \rangle
\quad \text{and} \quad  
 x_2 \in  \langle u_i;  i \ne 1 \rangle [f ^R ]  
=  \langle u_2  \rangle [f ^R ]   = f^{S - R }   \langle u_2  \rangle.
\]       
Similarly,
$ \alpha u_2    = y_1 + y_2 $, 
$ y_i \in   
\langle u_i \rangle$, $ i = 1,2 $,
and 
\[
 y_2 =   u_2 +  f  w_2,  \, w_2 \in
  \langle u_2 \rangle \quad \text{and} \quad  
  y_1\in   
\langle u_i;  i \ne 2 \rangle [f ^S ] 
= \langle u_1 \rangle [f ^ S] =  \langle u_1 \rangle   .
\] 
Then \
\begin{multline} 
\alpha  \tilde z= 
 ( u_1+ f v_1  +  x_2 ) + 
 (   f^{S -\tilde q } u_2 +  f^{S -\tilde q + 1  }w_2 +  f^{S -\tilde q } y_1 ) = \\
 \tilde z +(  f v_1  +  f^{S -\tilde q } y_1 ) 
 + (   x_2 +  f^{S - \tilde q + 1  }w_2 ) .
\end{multline}  
From  $ S - \tilde q > 0 $ follows 
$   f^{ S -\tilde  q   } y_1 \in f  \langle u_1  \rangle $. 
From  $ \tilde q > R $ follows  $S - R \ge S - \tilde q + 1$, and
therefore $
x_2 \in  f ^{  S - \tilde q +1  } \langle u_2\rangle  $. 
Set \,$
 \hat{v}  =   f  v_1 + f^{ S -\tilde q   } y_1$\,
and  \,$\hat{w} =  x_2+  f^{S -\tilde q +1} w_2 $.
Then
\beq \label{eq.smspq} 
 \alpha \tilde  z =  \tilde  z+
 \hat{v} + \hat{w},  \quad {\rm{and}} \quad
 \hat{v} \in f \langle u_1 \rangle, \,
\hat{w} \in   f^{S -\tilde q +1}  \langle u_2\rangle . 
\eeq
Define  \,$\tilde L  =  \langle \tilde z,   f^{S -\tilde q  + 1} u_2 \rangle $.
Because   of $ f \tilde z =  f  u_1 +  f^{S - \tilde q  + 1} u_2  $ 
we obtain  
 \,$
\tilde L   =  \langle \tilde   z,
 f u_1  ,  f^{S - \tilde q  + 1} u_2 \rangle $. 
Then  
\eqref{eq.smspq} implies 
$  \langle \tilde  z \rangle  ^c \subseteq \tilde L $. 
If 
\[ 
   \beta : (u_1 , u_2 ) \mapsto  (u_1  , u_2 + fu_2 ) \in {\rm{Aut}}(V,f) 
\]
then 
$    \beta  \tilde  z     =  \tilde z+ f^{S - \tilde q  + 1} u_2  \in   \langle
 \tilde z \rangle  ^c $. 
  Hence
\[  \beta \tilde  z    -  \tilde  z =    f^{S - \tilde q  + 1} u_2  \in   
\langle \tilde  z \rangle  ^c  ,
\]
and therefore 
$\tilde L   \subseteq \langle \tilde  z \rangle  ^c $, and  we obtain
\beq \label{eq.leqt}
 \tilde L =  \langle   \tilde  z \rangle  ^c = \langle \tilde   z,
 f u_1  ,  f^{S - \tilde q  + 1} u_2 \rangle  . 
\eeq

We determine the dimension of  $   \langle \tilde z \rangle  ^c $.
If $ R = 1 $ then $ f u_1 = 0 $ and 
 $ f \tilde   z =  f^{S  -  \tilde q  + 1} u_2 $.  Therefore 
$\langle \tilde z \rangle  ^c  =  \langle  \tilde z \rangle  $ and 
$\dim  \langle  \tilde  z \rangle  = \e( \tilde z) =  \tilde q $. 
If $ R > 1 $ then  $\langle \tilde z \rangle  ^c  =   \langle\tilde z ,
f u_1 \rangle $. 
Let 
\[
x \: =  \:  \sum  \nolimits _{\mu = 0} ^{R-1} c_{\mu} f^{\mu} u_1 
+ 
  \sum  \nolimits _{\mu = 0 } ^{\tilde q -1} c_{\mu} 
 f^{S -  \tilde q +\mu}   u_2  \: 
\in  \:  \langle  \tilde z \rangle .
\]
Then $ x \in  \langle f u_1 \rangle $ if and only if 
$ c_0 = \cdots = c_{\tilde q-1} = 0 $, that is, $x = 0 $. 
Hence  $  \langle \tilde  z \rangle \cap  
\langle f u_1 \rangle = 0 $, and 
\beq \label{eq.vuebs}
\dim  \langle \tilde  z \rangle  ^c  =    \dim \langle \tilde z \rangle + 
\dim \langle f u_1 \rangle   = \tilde q + (R -1).
\eeq

At this point we go back to   
$ X =\langle    z \rangle  ^c =
 f^{R-s} \langle \tilde  z \rangle $.
Then \eqref{eq.leqt} yields  \eqref{eq.sncw}.  
From \eqref{eq.vuebs} we obtain 
\begin{multline*}
 \dim X =   \dim  f^{R-s}  \langle   \tilde z \rangle + 
\dim  f^{R-s}   \langle f u_1 \rangle = 
 [ \tilde q - (R-s) ]  +
[ ( R- 1) - (R-s ) ]  = \\
s + q  -1, 
\end{multline*}
which proves \eqref{eq.dmax}. 
If $ s > 1 $ then $X = \langle   z \rangle \oplus 
\langle   f ^{R -s +1}  u_1  \rangle $
is a direct sum of $f$-cyclic subspaces of dimension $ q $ and 
$ s -1$, respectively.  
Hence, in that case,  the elementary divisors of $f_{|X } $ are
$ \lambda ^ q $ and $ \lambda ^{s-1} $. 
It is easy to see   that the case  $ s = 1 $ leads to $ X =  
 \langle   z \rangle $. 

We show next  that  $    f^{R -s } u_1  \notin   X $. 
 Suppose to the contrary that
 $  f^{R -s } u_1 \in X $. 
Then 
$ z  =  f^{R -s } u_1 +  f^{S - q} u_2  \in  X $    would imply
  $  f^{S - q} u_2 \in  X  $.
 Hence $  \langle  f^{R -s } u_1  \rangle  \oplus  \langle  f^{S - q} u_2 \rangle
\subseteq X $, and therefore 
$\dim X \ge s + q $, in contradiction to \eqref{eq.dmax}.
Hence   $  f^{R -s } u_1 \notin X$. 
Let $  \pi _1 $ be the projection of $V$ on $    \langle u_1   \rangle $ 
along  $    \langle u_2   \rangle $. 
Then  $ \pi _1 \in {\rm{End}}(V,f) $.
But $  \pi _1  z=   f^{R -s } u_1 \notin  X$. 
Hence  the subspace  $X$ 
is not hyperinvariant. 
\end{proof} 

\bigskip

\noindent
{\bf{Example~\ref{ex.sh} continued.}}
If $ ( R , S ) = (1, 3) $  
then  $(s, q ) = (1, 2 ) $ is the 
only solution of \eqref{eq.vrzs}. 
Then $ (R-s, S - q ) = (0 , 1 ) $, and  
 $ X =  \langle  f^0 e_1 + f^1 e_2  \rangle^c  $
is the  only characteristic  non-hyperinvariant subspace of  $V$.
According to 
\eqref{eq.nhnu} there are $6$ hyperinvariant subspaces in
$ V$.  \hfill  $\square$

\section{Separating repeated and unrepeated elementary divisors}
 \label{sct.eg}

According to  Shoda's theorem only unrepeated elementary divisors 
are rele\-vant for the existence  of  characteristic non-hyperinvariant
subspaces. In this section we examine this fact in more detail.
Let $ E $ and $G$ be invariant subspaces of $V$ and  assume 
\begin{multline} \label{eq.hnngt} 
V= E \oplus G \quad  
\text{and} \quad 
d( f,  t ) = d( f_{|E} , t  ) \quad \text{if} \quad  d( f,  t )  =  1 \quad 
 \text{and} \\  
 d (f, t ) = d( f_{|G}, t  ) \quad \text{if} \quad  d (f, t )  > 1. 
\end{multline}
Thus the unrepeated elementary divisors of $f$ are those of 
$  f_{|E}$
and the repeated ones are those of  $f_{|G} $. 
If \eqref{eq.hnngt}
 holds then there exists a generator tuple $ U = ( u_1, \dots , u_m) $
adapted to $E$ and $G$ 
such that 
\beq \label{eq.asgnt}
E  =  \langle  u_i ; \e(u_i) = t_i ; d(t_i) = 1 \rangle \quad {\rm{and}} 
\quad G = \langle  u_i ; \e(u_i) = t_i ; d(t_i) > 1 \rangle . 
\eeq  
We shall see that 
a characteristic subspace $X $ is  hyperinvariant   in $V$ 
if and only if  
$X \cap E $  is  hyperinvariant in $E$.
We first consider a general direct sum decomposition
of $V$.

\begin{lemma} \label{la.hncho} 
Let $V_1$ and $V_2 $ be 
 invariant 
subspaces of $V$ such that  $ V = V _1 \oplus V_2 $. 
\begin{itemize}
\item[\rm{(i)}]
 If $X \in {\rm{Chinv}}(V,f)$
then
$ ( X \cap V_i)   \in  {\rm{Chinv }}( V_i ,f_{|V_i}) $, $i = 1,2$.
\item[\rm{(ii)}]
If   $X \in {\rm{Hinv}}(V,f)$
then
$ ( X \cap V_i)    \in  {\rm{Hinv}}( V_i, f_{|V_i} )$, $i = 1,2$.
\item[\rm{(iii)}] 
A subspace  $X$ 
 is hyperinvariant if and only if
$X$ is characteristic and 
\beq \label{eq.smun}
X = ( X \cap V_1 )  \oplus ( X \cap V_2 )  \quad and 
\quad   X \cap V_i \in {\rm{Hinv}} (V_i ,f_{|V_i}) , \; i = 1,2.
\eeq
\end{itemize}
\end{lemma}

\begin{proof}
(i)
 Let $x_1 \in X \cap V_1 $ and    
$ g_1 \in  {\rm{Aut}} (V_1, f_{|V_1} )$.
To show that $ g_1(x_1) \in  X $
we extend $ g_1 $ to  an automorphism
 $ g \in {\rm{Aut}} (V, f)$ as follows.
Let $ \iota_2 $ be the identity map of $V_2$. 
Then  $ g =   g_1 +  \iota_2
 \in {\rm{Aut}}(V, f)$.
Therefore   $X \in {\rm{Chinv}}(V, f)$ implies
  $ g ( x_1 ) \in X $.
On the other hand
$  g ( x_1 ) =   g_1(x_1) \in  V_1 $.
Hence  $  g_1(x_1) \in X \cap V_1 $. 
\\
(ii) Let $ x_1 \in X \cap V_1$ and 
  $ h_1 \in {\rm{End}}(V_1, f_{ |V_1})$.
If   $ 0_2 $ is the zero
map on $V_2$ then    $h =  h_1 + 0_2  \in  {\rm{End}}(V, f) $.
An argument as in part (i) shows that \mbox{$  h_1(x_1) \in X \cap V_1$}. 
\\
(iii)  
Let $ V_i = \oplus _{\nu = 1} ^{m_i}  \langle u^{(i)}_{\nu} \rangle $, $ i = 1, 2$, 
be  decomposed into 
cyclic subspaces,  
and let 
$ U = ( u _j )_{j=1} ^m  \in \mathcal{U} $ contain the  vectors 
$  u^{(1)}_{\nu}$, $\nu = 1, \dots, m_1$, and $   u^{(2)}_{\nu}, 
\nu = 1, \dots , m_2 $.
Define $ X_i = X \cap V_i   $, $i = 1,2 $.
Suppose  $X$ is characteristic and  \eqref{eq.smun} holds.
Then \mbox{Lemma~\ref{la.hncho}(ii)}  and $V_1 \cap V_2 = 0 $ 
imply 
\[
 X_i    =   \oplus _{\nu = 1} ^{m_i}
(  X _i  \cap  \langle u^{(i)}_{\nu} \rangle ) = 
 \oplus _{j = 1} ^{m}
(  X _i  \cap  \langle u_{j} \rangle ) , \: i = 1,2. 
\]
Hence 
\begin{multline*}
X = X _1 \oplus  X_2 
=  
 \oplus _{j = 1} ^{m}\Big(   ( X _1   \cap  \langle u_{j}   \rangle )
+ ( X _2   \cap  \langle u_{j} \rangle  ) 
\Big) \\
\subseteq \oplus
_{j = 1} ^m [ (  X _1 +X _2 ) 
  \cap  \langle u_j \rangle ] 
= 
  \oplus
_{j = 1} ^m ( X  \cap  \langle u_j \rangle  )
 \subseteq X. 
\end{multline*}
Then $X $ satisfies 
 \eqref{eq.brhy},  and therefore $X$ is hyperinvariant. 
Using \eqref{eq.brhy} it is not difficult to see that  $X \in  {\rm{Hinv}} (V, f)  $
 implies  \eqref{eq.smun}.  
\end{proof}

We apply the preceding lemma to the decomposition 
$V = E \oplus G $. 
Let  $ \pi _E $   be the projection of $ V $ on $E$ along $G$, 
 and let 
$ \pi_G $  denote  the complementary projection such 
that $ \pi_E  +  \pi_G = \iota $.
Then  
 $\pi_E  f = f  \pi_E  $ and $\pi_G  f = f  \pi_G$. 
 Let $  (X \cap E)_{H|E}  $ denote the largest hyperinvariant
subspace (with respect to $ f_{|E} $) contained in
$E$. 

\begin{lemma} \label{la.brba} 
Let $ E $ and $G$ be subspaces of $V$ such that 
\eqref{eq.hnngt} holds.
Suppose   $X $ is  a characteristic subspace of $V$.  
\begin{itemize}
\item[\rm{(i)}] 
Then 
\,$\pi_E X  = X \cap E$\, 
and $   \pi_G X  = X \cap G$, 
and 
\beq \label{eq.dcmpo}
X = ( X \cap E )  \oplus  ( X  \cap G ) . 
\eeq
Moreover, 
\beq \label{eq.egch} 
  X \cap E  \in  {\rm{Chinv}} (E ,f_{|E})  \quad
and \quad  
X  \cap G \in 
  {\rm{Hinv}} (G ,f_{|G}) ,
\eeq 
and 
\beq \label{eq.lahpi} 
X _H =  (X \cap E)_{H|E} 
\oplus  (X \cap G  ) 
\eeq 
\item[\rm{(ii)}] 
The subspace $ X $ is hyperinvariant  in $V$ if and only if 
$   X  \cap E    $ is hyper\-invariant  in $E$. 
\end{itemize} 
\end{lemma}

\begin{proof} 
(i) 
Let $ U \in \mathcal{U} $ 
satisfy \eqref{eq.asgnt}.  
If  $x \in X $ then Theorem~\ref{thm.wogn}(i) yields 
 \[
\pi _G x  = \left(   \sum \nolimits 
_{   d(t_i) > 1 }   \pi _i   \right) x     = \sum   \nolimits 
_{  d(t_i) > 1 }   \pi _i   x  \in X .
\]
Hence $\pi _G X \subseteq X \cap G $, and therefore 
$ \pi _G X =  X \cap G $. 
Then $ x = \pi _G x +\pi _E x $  implies  $ \pi _E x   \in X $.
Thus we 
obtain  $ \pi _E X  = X \cap E $, and 
$ X = \pi _E X \oplus   \pi _G X $.

From Lemma~\ref{la.hncho}(i) 
we conclude that  $  X \cap E   $
 and $  X \cap  G  $ 
are   characteristic in $E$, respectively in  $G$.
According to  \eqref{eq.hnngt} 
the map $ f_{|G} $ has only repeated 
elementary divisors.  Hence Theorem~\ref{thm.vnpsa} 
implies that  
 $   X \cap G   $ is hyperinvariant in $G$. 
The description of $X_H$ in  \eqref{eq.lahpi} follows from \eqref{eq.larg}.

(ii)
The subspace   $  X \cap  G    $ is  hyperinvariant in $G$. 
We  apply Lemma~\ref{la.hncho}(iii).

\end{proof}

The assumption that $X$ is   characteristic is 
essential for 
\eqref{eq.dcmpo}.

\begin{example}
{\rm{
Let $V= \langle u_1, u_2, u_3 \rangle$
with $ \e(u_1) = 1, \:  \e(u_2) = \e(u_3) = 2$.
Then $ E = \langle  u_1 \rangle  $ and $G =  \langle u_2, u_3 \rangle$.
Set $X = \langle  u_1 + f u_2 \rangle $.  Then 
$ X \cap G =
0 $ and   $ X \supsetneqq ( X \cap E ) \oplus ( X \cap G )$. 
}}
\end{example}

\medskip 

In the following  we deal with invariant subspaces associated 
 to  subsets of unrepeated elementary divisors of $f$. 
Let $T$ be an  invariant subspace of $V$
such that  $ f_{|T}  $ has only 
unrepeated elementary divisors 
and such
that 
\[
V =  T \oplus V_2 \quad \text{for some} \quad    V_2  \in \Inv(V,f) ,
\]
and  $  f_{|T}  $  
and $ f_{|V_2}$   have no elementary  divisors  in common.
The subspace $T$ can also be characterized as follows.
There exists  a $   T_2  \in \Inv(V,f) $ such that
\beq \label{eq.tchrs}
 T \oplus T_2 = E  , \:\: \text{and}   \:\:   
 T_2 \oplus G = V_2 ,  \:\;
 E \oplus G = V  \:\: \text{and}  \:\:  
E   \:\: \text{and}  \:\:  G \:\: \text{satisfy}  \:\:    \eqref{eq.hnngt}. 
\eeq
 Let $ \pi _T $ be the projection 
on $T$ along $V_2$ and $ \pi _{V_2} $ be the complementary
projection. 
If $ Y \subseteq T $ then 
$Y ^ {c_T} $  denotes the characteristic hull of $Y $ with respect to
$T$,
\[
Y ^ {c_T} =  \big\langle \alpha _T y ; \: y \in Y , \; \alpha _T \in 
{\rm{Aut}}( T,  f|_T ) \big\rangle  . 
\]
A spin-off  from the following lemma is a proof of
Theorem~\ref{thm.pfle}(i).

\begin{lemma} \label{la.umbn}
Let $T$ be an invariant 
 subspace such that $f_{|T} $ has only unrepeated elementary 
divisors and \eqref{eq.tchrs}  holds. 
Suppose   $X $ is  a characteristic subspace of $V$.  
\begin{itemize}
\item[\rm{(i)}] 
Then $  X \cap T  \in  {\rm{Chinv}} (T,f_{|T})$. 
If  the subspace $ X $ is hyperinvariant  in $V$
then  \,\,$   X  \cap T    $\,  is hyper\-invariant  in $T$.
\item[\rm{(ii)}]  
We have 
\beq \label{eq.cplds} 
 {\rm{Aut}}(  T, f_{|T} )    = \{ \pi_T \alpha  \pi_T; \, \alpha \in
{\rm{Aut}}( V, f) \}. 
\eeq  
 \item[\rm{(iii)}] 
If $ Y \subseteq T $  
then
\beq \label{eq.nmgz} 
Y^c  \cap T = \pi_T Y^c  = Y ^ {c_T} .
\eeq
\end{itemize}
\end{lemma} 

\begin{proof} 
(i)  Because of  $ V = T \oplus V_2 $  
one can apply  Lemma~\ref{la.hncho}.
(ii) 
Let $ U = \{u_1 , \dots , u_m \}  $
  be a generator tuple
of $V$ such that  a subtuple 
$ U_T = \{ u_{\tau_1} , \dots , u_{\tau_q} \} \subseteq U  $
is a generator tuple of $T$ (with respect to $f_{|T}$). 
Set  $I_T = \{\tau _1, \dots , \tau_q \}$.
If $ i \in  I_T  $
then Lemma~\ref{la.unrpes} implies 
\[ 
\alpha \, u_i =c_i  u_i + fv_i +  
\sum \nolimits _{j  \in  I_T, \,  j \ne i} w_j +
 \sum \nolimits _{k \notin  I_T  , \,  1 \le k \le m  } x_k ,
 \]
where 
$c_i \ne 0 $,
 $v_i \in \langle u_i \rangle$, 
$w_j  \in  \langle u_j \rangle[f^{t_i}]$
and $ x_k \in \langle u_k  \rangle [f^{t_i}]$.
Hence 
\[
( \pi_T \alpha )  u_i = c_i u_i + fv_i + \sum \nolimits _{j\in I_T, j \ne i}w_j  .
\]
Then Lemma~\ref{la.mxmn}    yields 
  $\e \!\big( ( \pi_T \alpha) u_i  \big) = t_i $, and 
\[
 \h( f^{t_i - 1 } ( \pi_T \alpha )  u_i )  =
\h(  f^{t_i - 1 }  u_i ) =  t_i - 1  , 
\]
and  
$ \h(  ( \pi_T \alpha )  u_i ) = 0$. 
By Lemma~\ref{la.jtgn}   the element  $(\pi_T \alpha) u_i  \in T  $
is a generator. 
Hence   $ \big(  (\pi_T \alpha )u_{\tau _1} , \dots , 
( \pi_T \alpha)u_{\tau _q}  \big) $
is a
generator tuple  of $T $.  Then 
the map  given by 
\[ 
\pi_T\alpha  
:  u_i \mapsto (  \pi_T \alpha ) u_i   , \; i \in I_T, 
\]
is in  $ {\rm{Aut}}( T,  f|_T )  $. 
Hence $  \pi_T \alpha   \pi_T \in {\rm{Aut}}( T,  f|_T ) $. 
Now consider an automorphism  
 $ \alpha _T  \in {\rm{Aut}}( T,  f|_T ) $.  
We extend    $ \alpha _T $ 
to an automorphism $ \tilde{\alpha} 
\in {\rm{Aut}}( V,  f )$
defining
\beq \label{eq.stxt}  
\tilde{\alpha}    : u_i \mapsto \alpha _T u_i \quad  \text{if}  \quad   i \in I_T , 
\quad  \text{and} \quad 
\tilde{\alpha}     : u_i \mapsto   u_i \quad  \text{if}  \quad   i \notin I_T.
\eeq 
Then $ (\pi _T \tilde{\alpha}  \pi _T ) u_{\tau_i }=   \alpha _T  u_{\tau_i } $, 
$ i = 1, \dots , q$, and  therefore $  \alpha _T = \pi _T \tilde{\alpha}  \pi _T $.

(iii) 
Let us show first that 
 $ \pi_T  Y^c =     Y^c \cap T  $. 
%
If $ \alpha \in {\rm{Aut}}( V,  f )$ then 
$\alpha _T = \pi _T \alpha \pi _T \in {\rm{Aut}}( T,  f_{|T } ) $.
Let  $\tilde{\alpha}  \in {\rm{Aut}}( V,  f ) $ be  the extension of  $\alpha _T$ 
given  by \eqref{eq.stxt}.
If  $ y \in Y  \subseteq T $
then 
$ \pi _T \alpha  y = \pi _T \alpha \pi _T y = \alpha _T y 
 = 
\tilde{\alpha} y \in Y^c $.
Hence  $  \pi_T  Y^c \subseteq  Y^c $, 
which suffices to prove  $ \pi_T  Y^c =     Y^c \cap T  $. 
Then 
\begin{multline*} 
Y ^c \cap T =   \pi _T Y^c   
\\
\langle ( \pi _T  \alpha ) y ; \, \alpha \in
{\rm{Aut}}( V,  f )  ,  \,  y \in Y \rangle =
\langle ( \pi _T  \alpha \pi _T ) y ; \, \alpha \in
{\rm{Aut}}( V,  f )  ,  \,  y \in Y \rangle  
=
\\
 \langle (   \alpha  _T ) y ; \, \alpha _T \in
{\rm{Aut}}( T,  f_{|T}  )  ,  \,  y \in Y \rangle   = Y ^{c_T} .
\end{multline*}  
\end{proof}

\noindent
{\bf{Proof of Theorem~\ref{thm.pfle}(i).}}
We apply Lemma~\ref{la.umbn}.
Let $ T = \langle u_\rho ,   u_\tau \rangle $. 
Set $ z_{s,q}  =  f^{R -s } u_\rho  +  f^{S -q } u_\tau $, 
and $ Y =  \langle z_{s,q}  \rangle   $.
Then $Y \subseteq T $, and 
$  Y^{c_T} = Y ^c \cap T $.  It follows from  Proposition~\ref{la.rgwpo}
that  $ Y^{c_T} $ is
not hyperinvariant  in $T$. Hence 
$Y^c =  \langle z_{s,q}  \rangle  ^c $ is not  hyperinvariant  in $V$.
\hfill $\square$ 

Suppose $Y \subseteq E $ is characteristic and 
non-hyperinvariant with respect to $f_{|E} $
 How can one extend 
$Y$  to a  subspace $ X $ that has these properties 
 in  the entire space $V$?

\begin{theorem}     \label{thm.crgna} 
Let $ E $ and $G$ be subspaces of $V$ such that 
\eqref{eq.hnngt} holds.
Let 
$Y  \in  {\rm{Chinv}} (E, f_{|E})  $.
If  $Y_s $ is  a subspace of $Y$ and
 $ W $ is 
 a subspace of  $G $ such that 
$  Y_s + W $ is characteristic  in $V$ 
then 
\beq \label{eq.cdroa}
  (  Y + W  ) ^c   \cap E = Y  .
\eeq 
The subspace $   (  Y + W  ) ^c $
is hyperinvariant in $V$  only if $ Y $ is hyperinvariant  in~$E$.  
\end{theorem}

\begin{proof} 
Set  $ X = (Y +  W  ) ^c $. We are  going to show that 
$ \pi _E X  = Y $. 
 Because of 
$ Y \subseteq E $ and  $ Y  \subseteq X $  we see  that
\beq \label{eq.mdsn}  
 Y  \subseteq ( X \cap E )\subseteq \pi _E X  .
\eeq 
  It is obvious that $  (Y + W)^c  \subseteq Y^c + W^c $.
From  $Y, W \subseteq (Y+W)^c$ follows 
 $ Y^c + W^c \subseteq (Y + W)^c$.  Hence 
$ X = Y^c + W^c $.  Therefore 
\beq \label{eq.smtp} 
\pi _E X =  \pi _E  Y^c +  \pi _E  W^c . 
\eeq 
Since $Y_s + W $ is characteristic we have 
$W^c \subseteq ( Y_s + W)  ^c  =  Y_s + W $.
 Hence  $Y_s \subseteq E $ and 
$ W \subseteq G $, or equivalently 
$  \pi _E W  = 0 $, imply  
$\pi _E W^c   \subseteq  Y_s   \subseteq Y $. 
Since  $Y$ is characteristic in $E$ it follows from 
Lemma~\ref{la.umbn}  that 
$ Y ^{c_E} =  Y = \pi_E Y^c $.  Hence \eqref{eq.smtp}  implies 
$  \pi _E X \subseteq  Y  $, and \eqref{eq.mdsn} yields 
$ Y =   X  \cap E  =  \pi_E X    $. 
Thus we  proved \eqref{eq.cdroa}.
If $ Y $ is not hyperinvariant in $E$ then  it follows from 
Lemma~\ref{la.brba}  that $X $ is not hyperinvariant in $V$.  
\end{proof}

\section{The main theorem}  \label{sct.mn} 
 
In this section we drop the assumption that $V$ is 
generated by two vectors.  Hence 
the map $f$  can have more than two elementary divisors.
We assume that only two of them are unrepeated.
In that case we obtain a description of the family of
characteristic non-hyperinvariant subspaces 
which extends Theorem~\ref{thm.pfle}(ii).

We first consider  the case where $f$ has no unrepeated  elementary
divisors. 
In terms of a decomposition $ V = E \oplus G $ in \eqref{eq.hnngt}   
that assumption 
is equivalent to $V = G $. 

\begin{lemma} \label{la.elnsn} 
Let 
$ U = (u_1, \dots, u_m ) \in \mathcal{U} $, $ \e(u_i) = t_i $, $i = 1,
\dots , m$. 
Suppose  $f$ has no unrepeated elementary divisors, that is 
\beq \label{eq.smnt} 
d(t_i) > 1 \quad  for \: \: all  \quad i=1, \dots , m.
\eeq 
Then  $X \subseteq V $   is hyperinvariant if and only if 
\beq \label{eq.sgnrt} 
X =   \langle  f^{r_1 }  u_1 + \cdots +  f^{r_m } u_m \rangle ^c 
\eeq 
for some $ \vec r = ( r_1 , \dots , r_m) $ satisfying 
\eqref{eq.zwrug}.
\end{lemma}

\begin{proof} 
By  Shoda's theorem 
the assumption 
 \eqref{eq.smnt}  implies that each characteristic
subspace of $V$ is hyperinvariant. 
Hence, if $X$ is  of the form \eqref{eq.sgnrt} then 
$ X \in  {\rm{Hinv}} (V,f)$.
Now suppose  $X  $ is  hyperinvariant.
According to Theorem~\ref{thm.knwr} 
we have 
\beq \label{eq.xwrl} 
X =  W(\vec r \, ) =
   \langle f ^{r_1}  u_1 \rangle  \,  \oplus  \,  \, 
 \cdots  \,  \,
\oplus 
  \langle  f ^{r_m} u_m \rangle  
\eeq
for some $ \vec r $ satisfying \eqref{eq.zwrug}.
Define 
$w = f^{r_1 }  u_1 + \cdots +  f^{r_m } u_m $.
Then $w \in W(\vec r \, ) $. Since $W(\vec r \, )$ is hyperinvariant, it is obvious 
that  $\langle w \rangle^c \subseteq W(\vec r\, )$.
To prove  the converse inclusion 
we apply  Theorem~\ref{thm.wogn}(ii) to the hyperinvariant subspace 
$\langle w \rangle ^c $. We obtain 
$\pi _i w =  f ^{r_i}   u_i  \in \langle w \rangle^c $, 
$ i = 1, \dots , m$,  
and therefore  $ W(\vec r \, )  \subseteq \langle w \rangle^c $.
Hence $ X =  W(\vec r \, )  =  \langle w \rangle^c $. 
\end{proof} 

The notation in Theorem~\ref{thm.fnla}
below will be   the following.
We write 
$ \vec{\mu} ' = \vec{\mu} + (\vec e_\rho + \vec e_\tau)  $
if 
\begin{multline} \label{eq.mump}
(\mu ' _1, \dots , \mu ' _\rho , \dots , \mu ' _\tau , \dots , \mu ' _m )
= \\ 
(\mu_1, \dots , \mu_\rho , \dots , \mu _\tau , \dots , \mu_m )
+ ( 0 , \dots , 1  , \dots , 1  , \dots , 0 ),
\end{multline}  
that is, 
\beq \label{eq.nwanar} 
\mu '_i = \mu _i \quad  \text{if}  \quad  i \ne \rho , \tau,  \quad  \text{and}
\quad \mu' _{\rho } = \mu _{\rho} + 1 ,  \; 
\mu'_{\tau} = \mu _\tau + 1.  
\eeq 
Then  $ \vec{\mu} \,' \in  \mathcal{L}(\vec  t \, \, ) $ 
means
\beq \label{eq.agfrs} 
  \mu '_i \le t_i ,  \:  i = 1, \dots , m , 
 \:  \: \text{and} \: \:  \:
\mu '_1 \le \cdots \le \mu' _m ,  \:  \: \: 
 t_1 - \mu '_1 \le \cdots \le t_m - \mu' _m . 
\eeq 

\begin{theorem} \label{thm.fnla}
 Let $|K| = 2 $. 
Suppose  that among the  elementary divisors of 
 $f$ 
there are  exactly two unrepeated ones, namely 
$\lambda ^R $ and $\lambda ^S $. 
 Let
 $ U = (u_1 , \dots , u_m ) \in \mathcal{U} $ and 
$ \e(u_\rho ) =
R $, $ \e(u_\tau ) 
=  S $.
A subspace  $X \subseteq V $ is  characteristic and not hyperinvariant
if and only if 
\beq \label{eq.mngr}
X = \langle f^{\mu _1 } u_1 + \cdots + f ^{\mu _m} u_m  \rangle ^c ,
\eeq
such that the entries  $ \mu_{\rho} $ and $ \mu_{\tau} $  of 
$ \mu = (\mu_1 , \dots , \mu_m)$ 
satisfy 
\beq \label{eq.vnszq} 
0 \le \mu _\rho < \mu _\tau \quad and \quad
 0 <  R - \mu_ \rho < S - \mu _\tau , 
\eeq 
and such that 
$ \vec{\mu} + (\vec e_\rho + \vec e_\tau) \in  \mathcal{L}(\vec  t \, \, ) $.
\end{theorem}

\medskip 

\begin{proof} 
Set  $E = \langle u_\rho \rangle \oplus   \langle   u_\tau  \rangle  $ and 
$ G = \langle u_i; \e(u_i) = t_i; d(t_i) > 1 \rangle$. 
Then $ V =  E  \oplus G$,   as in \eqref{eq.hnngt}. 
Suppose   $ X \in {\rm{Chinv}} (V,f) \setminus  {\rm{Hinv}} (V,f) $.
Lemma~\ref{la.brba} implies  
\[
 X = (  X \cap E ) \oplus (X \cap G ) . 
\]  
The subspace  $ X \cap G $
 is hyperinvariant in $G$,  whereas 
$  X \cap E $ is  characteristic but not hyperinvariant in  $E$. 
By assumption  $E$ is generated by two elements. 
Hence it follows from Proposition~\ref{la.allen} that
$ X \cap E  = \langle z \rangle^{c_E} $.
Referring  to  \eqref{eq.ksptu}  and \eqref{eq.cnglu}
we obtain 
$ z = f^{\mu _\rho  } u_\rho + f^{ \mu _\tau  } u_\tau $  
   for some integers  $\mu _\rho , \mu _\tau  $ satisfying 
\eqref{eq.vnszq}.
According to  Lemma~\ref{la.elnsn} 
we have  \,$ X \cap G = \langle w \rangle^{c_G}$ 
for some 
\[ w  = 
\sum _{1 \le i \le m ; \,\,   i \ne \rho, \tau} f^{\mu _i } u_i  \in  X
\cap G  .  
\]
Set  $ \hat{z} =  z  + w $. 
%
Let us show that $X =  \langle \hat{z} \rangle^c $.
If $   \alpha _E \in
{\rm{Aut}}(E, f_{|E} ) $ and $ \alpha _G \in {\rm{Aut}}(f_{|G}, V ) $
then $   \alpha_ V  =  \alpha _E +  \alpha _G \in {\rm{Aut}}(f, V) $. 
Hence 
\,$
 \alpha _E   z +  \alpha _G w  =  \alpha_ V (  z  + w  ) 
\in  \langle \hat{z} \rangle^c   $,
and we obtain 
\beq  \label{eq.npst}  
X =
\langle z \rangle^{c_E}  \oplus 
\langle w \rangle^{c_G} 
\subseteq  \langle \hat{z} \rangle^c . 
\eeq 
Since $ X $ is characteristic 
it is obvious that 
$ \langle \hat{z} \rangle^c \subseteq X$.  Therefore 
$X = \langle \hat{z} \rangle^c  $,  and $X $  is of the form
\eqref{eq.mngr}. 
From \eqref{eq.hxnru}   
follows 
\[  
 \Big( \langle z\rangle^{c_E}  \Big)_{H|E} 
   = 
\langle   f ^{\mu _\rho + 1}u_\rho \rangle \oplus \langle  f ^{\mu
_\tau + 1} u_\tau \rangle .
\]
The proof of Lemma~\ref{la.elnsn} shows that 
\[
\langle w \rangle^{c_G}
= 
\underset{1 \le i \le m; \; i \ne \rho , \tau }
{ \oplus } \langle f ^{\mu _i} u_i  \rangle .
\]
Then \eqref{eq.lahpi}  yields 
\beq \label{eq.vngob}
X_H =  \Big( \langle z \rangle^{c_E}  \Big)_{H|E}  
 \oplus  \big\langle w  \big\rangle ^{c_G}
=
\underset 
{i = 1, \dots , m }{\oplus} \langle f ^{\mu' _i} u_i  \rangle , 
\eeq 
with  $ \mu'_i $, $i = 1, \dots , m$,  defined by  \eqref{eq.nwanar}. 
The subspace $X_H $ is hyperinvariant. 
Hence Theorem~\ref{thm.knwr} implies 
$\vec {\mu '} =  ( \mu' _1 , \dots , \mu '_m ) \in 
   \mathcal{L}(\vec  t \, \, ) $.

Now consider a subspace  
\[ 
X = \Big\langle \sum \nolimits  _{i=1} ^m   f^{\mu _i } u_i  \Big\rangle ^c
\]
 assuming  that  the  inequalities 
\eqref{eq.vnszq} hold  and 
 $ \vec{\mu}\, '  = \vec{\mu} + (\vec e_\rho + \vec e_\tau) 
  \in  \mathcal{L}(\vec  t \, \, ) $.  
Define 
\[
 z   =  f^{\mu _\rho } u_\rho + f^{\mu _\tau } u_\tau ,  
\:\:
 w = \sum _{  1 \le i \le m ; \:\,   i \ne \rho , \tau }    
f^{\mu _i } u_i ,  \quad \text{and}
 \quad   \hat{z} = z  + w . 
\]
Then $X = \langle  \hat{z} \rangle  ^c $. 
Since $X$ is characteristic it follows from  Lemma~\ref{la.brba} 
 that
$ \pi _E  \hat{z} = z \in X  \cap E $, 
 $ \pi _G  \hat{z} = w  \in X \cap G $.
Set $Y =  \langle  z  \rangle  ^{c_E}  $ and $W =
 \langle  w  \rangle  ^{c_G} $.
Then Lemma \ref{la.umbn}(iii) implies 
$ Y =  \langle  z  \rangle  ^c \cap E $.  
The inequalities \eqref{eq.agfrs}  ensure that 
$ W $ is   hyper\-invariant in $G$ and  
that  $ W = \oplus _{i \ne \rho , \tau }  \langle  f^{\mu _i } u_i  \rangle$.  
Let us show first  that $ X = ( Y+ W )^c  $.
Since $X$ is characteristic we have $  ( Y+ W )^c  \subseteq X $. 
Conversely,  $  z + w \in  Y+ W  $ yields 
 $ X =   \langle  z  + w  \rangle  ^ c  \subseteq   ( Y+ W )^c $.
It follows from \eqref{eq.vnszq}  
 that
 $ \langle  z  \rangle  ^{c_E} =   Y$ is not hyperinvariant  
 in $E$. Hence 
\[ 
Y_{H|E}  = 
\langle     f^{\mu _\rho + 1} u_\rho , \,  f^{\mu _\tau +  1 } u_\tau   \rangle 
=\langle     f^{\mu' _\rho } u_\rho , \, f^{\mu' _\tau } u_\tau   \rangle .
\]
Then  $ \vec{\mu}\, ' 
  \in  \mathcal{L}(\vec  t \, \, ) $ 
 implies that  
$Y_{H|E}  + W  $ is hyperinvariant in $V$. 
We apply Theorem~\ref{thm.crgna} choosing  $Y_s =  Y_{H|E} $
and conclude that $X$ is not hyperinvariant. 
\end{proof}

\begin{example}
 {\rm{
We consider a map $f$ with elementary divisors 
$ \lambda ,  \lambda ^3 ,\lambda ^7,  \lambda ^7$
and 
$V = \oplus _{i = 1} ^4  \langle  u_i \rangle  $ such
that 
$ \vec t =  \big(\e(u_i)  \big) 
= (1, 3, 7, 7 )$. 
We know that 
\[0 \le \mu _1 < 1 , \,   0 \le \mu _2 < 3  \quad \text{and}
 \quad  \mu _1 < \mu _2 ,  \;\:  1 - \mu_1 < 3 - \mu_2 , 
\] 
has the  unique solution $ (  \mu _1 , \mu _2 ) = (0, 1) $. 
Hence  
\[
 \vec{\mu} = (0, 1, \mu_3, \mu_4)  \quad \text{and} 
\quad 
 \vec{\mu}\,'  = (1, 2, \mu_3, \mu_4) .
\]
Then $ \vec{\mu}\, ' 
  \in  \mathcal{L}(\vec  t \, \, ) $ 
if and only if 
$  (\mu_3, \mu_4) = (j, j ) $, $ j = 2, \dots , 6 $.
Define $ g_j = u_1 + fu_2 + f^j u_3 +f^j u_4 $
and 
$ X_j = \langle  g_j   \rangle  ^c $,  $ j = 2, \dots , 6 $,
and $ X = \langle  u_1 + fu_2 \rangle ^c $. 
Then 
\[  
    \e(u_2 + f ^{j-1} u_3 + f^{j-1}  u_4 ) =  \e(u_2) , \:  j = 5, 6 , 
\]
implies
$ X_5 = X_6 =  X $.  
We obtain $4$ different characteristic not hyper\-invariant subspaces,
namely $ X , X_2, X_3 , X_4$. 
There are $n_H = 30 $ hyperinvariant subspaces in $V$. 
}} 
\end{example}

\subsection{Concluding remarks}

We have restricted our study to the case of two unrepeated elementary divisors.
Using  methods of our paper  one can prove more 
general results. 
For example one can extend 
Theorem~\ref{thm.pfle}(i) as 
follows.  

\begin{theorem}  \label{thm.wghpr} 
Let $ |K | = 2 $.
Suppose 
$ \lambda ^{t_{\rho _1}} , \dots , \lambda ^{t_{\rho _k}} $
are unrepeated elementary divisors of $f$  such that 
$ t_{\rho _1} < \cdots < t_{\rho _k} $.
Let $ U = ( u_1, \dots , u_m )  \in \mathcal{U} $
be a generator tuple with 
$ \e( u_{\rho _j })  = t_{\rho _j} $, $ j= 1, \dots , k $. 
Suppose 
$(\mu_{ \rho_1}  , \dots , \mu_{\rho_k}  ) \in \N_0^k $
and  set 
$ z =  f^{ \mu _{\rho_1}  }  u_{\rho _1 } 
+ \cdots +f^{ \mu _{\rho _k } }    u_{\rho _k} $.
If    
\beq \label{eq.wzto} 
 \mu_{\rho_j}  <  t_{\rho _j} ,  \:  j = 1, \dots , k ,  
\eeq
and  
\beq \label{eq.mros}
\mu_{ \rho_1}  < \mu_{ \rho_2} 
 < \cdots < \mu _{\rho _k }    \quad \text{and}
 \quad 
0 < t_{\rho_1} - \mu_{ \rho_1}  < \cdots < t_{\rho _k } - \mu _{\rho _k }, 
\eeq 
then 
$ X = \langle z \rangle ^c $ is a characteristic and not hyperinvariant
subspace of $V$.
\end{theorem}

In  Example~\ref{ex.blo}  below we    illustrate
Theorem~\ref{thm.wghpr}.
We also construct   a  non-hyperinvariant subspace 
that is not   the characteristic 
hull  of a single element. Such subspaces will be 
studied in a subsequent paper. 

\begin{example}  \label{ex.blo}
{\rm{
Let $ |K | = 2 $. 
Suppose  $ f$ has  elementary divisors $ \lambda , \lambda ^3, \lambda ^5 $
such that 
$ V =  \langle u_1 \rangle   
  \oplus \langle u_2 \rangle \oplus  \langle u_3 \rangle   $
 and 
\[
(t_1, t_2 , t_2 ) = 
 \big( \e(u_1 ),   \e(u_2 ) , \e(u_3 ) \big) =
(1,3,5) .
\] 
Then 
$(\mu_1, \mu_2 , \mu_3) =( 0,1,2)$
is the only solution 
of  the set of inequalities 
\begin{multline*}
0 \le \mu_1 < 1, \;  0 \le \mu_2 <  3, \,  0 \le \mu_3 <   5 ; 
\\ 
0 \le \mu_1 <\mu_2  < \mu_3 , \:\:  0 < 1 - \mu_1 < 3 -  \mu_2  < 5 - \mu_3  .
\end{multline*}
Set 
\beq \label{eq.zfbz} 
z =  u_1 + fu_2 + f^2 u_3 . 
\eeq 
 Then  $ X = \langle z \rangle^c 
\in {\rm{Chinv}}(V,f) \setminus {\rm{Hinv}}(V,f) $.
The fact 
that $X$ is not hyperinvariant can be verified as follows. 
We have $\e(z) = 3 $ and the indicator sequence of $z$
is 
 $  H(z) = (0, 2,  4, \infty, \infty)$.
Define $Y = \{ x \in V \ | \ H(x) = H(z) \}$.
Then
\[
Y = \{ z + v_2 + v_3 ; \:\,  v_2 \in  f^2 \langle u_2  \rangle ; \: 
v_3 \in  f^3 \langle  u_3 \rangle \} .
\] 
Hence
 \beq \label{eq.zyua}
   \langle z \rangle^c = \langle Y \rangle =
 \langle z,  f^2 u_2  ,  f^3 u_3 \rangle .
\eeq 
Then  $ \pi _3  z =  f^2 u_3   \notin \langle Y \rangle $.
 Therefore  $ \langle Y \rangle $ is not hyperinvariant in $V$.

The following subspace $W$  is also in 
$ {\rm{Chinv}}(V,f) \setminus {\rm{Hinv}}(V,f) $. We shall see
that it is not the characteristic hull of a single element. 
Let 
\[
 W = \langle z_ 1 ,  z_2   \rangle^c  \quad \text{and} \quad 
z_1 = u_1 + fu_2 ,  \:  z_2 = f^2 u_3 .
\]
Then 
  $ \langle z_ 1 , f^3 u_3 \rangle^c 
= \langle z_ 1  \rangle $ and  
$ \langle z_ 2  \rangle^c 
= \IIm f^2 \cap \Ker f^3 =  \IIm f^2 =   \langle f^2 u_2 , f^2 u_3  \rangle  $.
Therefore 
$ W =  \langle z_ 1  \rangle \oplus \langle   f^2 u_3     \rangle  $.
We have 
$ \pi _1 z_1 = u_1 \notin W $. Hence 
$W$ is  not hyperinvariant. 
Suppose $ W =  \langle w  \rangle^c $ for some $ w \in W$.
Then $ w = x + y $ with  $ x \in \langle z_ 1  \rangle $,  
$ y \in  \langle  f^2 u_3  \rangle   $. 
Because of $ z_1 \in W $ and $ \h(z_1 ) = 0 $
it is necessary that 
$ \h(w) = 0 $. Hence  $ \h(x) = 0 $.
Therefore 
$\langle x \rangle = \langle z_1 \rangle $, and 
$ x = \gamma z_1 $ for some $\gamma \in {\rm{Aut}}(f, V )$.
Hence we can assume
$ w = z_1 + y $. 
If $ \h(y ) \ge 3 $ then $ \e(w) = \e(z_1)$ and 
$  \langle w   \rangle ^c  =  \langle z_ 1  \rangle ^c \subsetneq W $.
If  $ \h(y ) = 2 $ then $ y = \beta(  f^2 u_3 )  $ for some $\beta 
\in {\rm{Aut}}(f, V )$,
and therefore  
\[  \langle w   \rangle ^c  =  \langle z_1 +   f^2 u_3   \rangle ^c 
= \langle z  \rangle ^c ,
\]
where $z$ is given by \eqref{eq.zfbz}. 
We have seen before that $ \pi_3 z =  f^2 u_3 \notin  \langle z  \rangle ^c $.
Hence  
$  \langle w   \rangle ^c  =  \langle z_ 1 + z_2  \rangle ^c \subsetneq W $.
Therefore $  \langle w   \rangle ^c  \ne W $ for all $w \in W $.
\hfill $\square$ 
}}
\end{example}

\noindent
{\bf{Acknowledgement:}} 
We are grateful to a referee for   helpful suggestions and comments.

\end{document}